\documentclass[11pt]{amsart}

\usepackage{a4wide,amssymb}

\usepackage{upref}

\newcommand{\CC}{\mathbb{C}}
\newcommand{\NN}{\mathbb{N}}
\newcommand{\RR}{\mathbb{R}}
\newcommand{\ZZ}{\mathbb{Z}}
\newcommand{\id}{\mathrm{id}}
\newcommand{\abs}[1]{\left|#1\right|}

\newcommand{\sga}[3]{\left[ \begin{smallmatrix}
    #1 \\ #2 \end{smallmatrix} \right]_{#3}}
\newcommand{\gabi}[3]{\begin{bmatrix}
    #1\\  #2 \end{bmatrix}_{#3}}
\newcommand{\fix}{\mathrm{fix}}
\newcommand{\orb}{\mathrm{orb}}
\newcommand{\diri}[2]{D^{}_{#1}(#2)}
\newcommand{\zetfun}[1]{\zeta\left(#1\right)}
\newcommand{\ts}{\hspace*{0.5pt}} 

\newtheorem{theorem}{Theorem}[section]
\newtheorem{lemma}[theorem]{Lemma}
\newtheorem{prop}[theorem]{Proposition}
\newtheorem{corollary}[theorem]{Corollary}
\theoremstyle{definition}
\newtheorem*{rem}{Remark}

\begin{document}
\title[Fixed point and orbit counts]
{A note on the relation between \\[2mm]
fixed point and orbit count sequences}
\author{Michael Baake}
\author{Natascha Neum\"arker}
\address{Fakult\"at f\"ur Mathematik, Universit\"at Bielefeld, 
Postfach 100131, 33501 Bielefeld, Germany}
\email{$\{$mbaake,naneumae$\}$@math.uni-bielefeld.de}

\begin{abstract}
  The relation between fixed point and orbit count sequences is
  investigated from the point of view of linear mappings on the space
  of arithmetic functions. Spectral and asymptotic properties are
  derived and several quantities are explicitly given in terms of
  Gaussian binomial coefficients.
\end{abstract}
\maketitle


\section{Introduction and general setting}
Each map of an arbitrary set into itself gives rise to sequences
$a=(a_n)_{n\in\NN}$ and $c=(c_n)_{n\in\NN}$ that count the number of
periodic points with period $n$ and the orbits of length $n$,
respectively.  In many interesting cases, $a_n$ and $c_n$ are finite
for all $n\in\NN$, turning $a$ and $c$ into sequences of
non-negative integers.  They are then related by M\"obius inversion as
\begin{equation}\label{intTrans}
     a_n = \sum_{d|n} d\, c^{}_d \quad \text{and}\quad 
     c_n = \frac{1}{n}\sum_{d|n}\mu\left(\frac{n}{d}\right) a^{}_d.
\end{equation}
Here, $\mu$ denotes the M\"obius function from elementary number
theory \cite{apostol}.  Puri and Ward \cite{PuriWard} study these
relations as paired transformations between certain non-negative
integer sequences, with special emphasis on arithmetic and asymptotic
aspects.  It is the purpose of this short note, which is based on
\cite{NN}, to extend these transformations to the space of arithmetic
functions, on which they act as linear operators, and to investigate
their spectral properties and their asymptotic behaviour under
iteration, thus generalising results from \cite{PuriWard}.

Let $\mathcal{A} = \CC_{}^{\NN}$ denote the set of arithmetic
functions. Its elements can be thought of either as functions $f\!
:\,\NN\to\CC, n\mapsto f(n)$, or as sequences of complex numbers
$(f_n)_{n\in\NN}$, hence the notations $f(n)$ and $f_n$ are used in
parallel; see \cite{apostol} for details and general notation.
Clearly, $\mathcal{A}$ is a complex vector space with respect to
component-wise addition and the usual scalar multiplication.
Additionally, it is equipped with the structure of a commutative
algebra $(\mathcal{A},+,\ast)$ with unit element, where $\ast$ denotes
\textit{Dirichlet convolution}, defined by
\[
    \bigl( f\ast g \bigr) (n) = \sum_{d|n} f(d) \, 
           g\left(\frac{n}{d}\right)
\]
for all $n\in\NN$.  The unit element with regard to this product is
$\delta_1 = (1,0,0,\ldots)$.  When $f\in\mathcal{A}$ is
invertible, the notation $f^{-1}$ refers to its \textit{Dirichlet
  inverse}, defined by $f\ast f^{-1} = f^{-1}\ast f = \delta_1$.
Furthermore, $f\cdot g$ denotes the pointwise product where 
$\bigl( f\!\cdot g\bigr)  (n) = f(n)\ts g(n)$ for all $n\in\NN$.  
Motivated by the transformations \eqref{intTrans}, we define
`$\fix$' and `$\orb$' as mappings from $\mathcal{A}$ into itself 
by $a\mapsto a_{}^{\prime}$, with
\begin{equation}
    a^{\prime}_n = \fix(a)_n = \sum_{d|n} d\, a^{}_d\quad 
    \text{and}\quad 
    a^{\prime}_n = \orb(a)_n = \frac{1}{n}\sum_{d|n} 
    \mu{\left(\frac{n}{d}\right)} \ts a^{}_d,
\end{equation}
respectively, for all $n\in\NN$.  Denoting by $N$ the arithmetic
function $n\mapsto n$, $\fix$ and $\orb$ can be written as
\begin{equation}\label{convRepres}
     \fix(f) = N\cdot \Bigl(\frac{1}{N}\ast f \Bigr)
     \quad \text{and}\quad 
     \orb(f) = \frac{1}{N}\cdot \bigl( \mu\ast f \bigr).
\end{equation}
This is particularly useful when considering the Dirichlet series of
$\fix(f)$ and $\orb(f)$.  In general, the Dirichlet series $D_f$ of
$f\in\mathcal{A}$ is defined by
\[
     D_f (s) = \sum_{m=1}^{\infty}  \frac{f(m)}{m^s} \ts ,
\]
which is viewed as a formal series at this stage.  This generating
function gives rise to the (formal) identities \cite{apostol}
\[
       D_{f+g}   =   D_f + D_g    \quad \text{and}\quad 
       D_{f\ast g}  = D_f \cdot D_g \ts .
\]
Using these in connection with \eqref{convRepres}, a straightforward
calculation shows
\begin{lemma}\label{dirichFix}
   Let $D_a(s)$ denote the Dirichlet series of $a\in\mathcal{A}$. Then,
\[
   \diri{\fix(a)}{s} = \zetfun{s} \diri{a}{s-1} 
   \quad \text{and} \quad
   \diri{\orb(a)}{s} = \frac{1}{\zetfun{s+1}} \, \diri{a}{s+1} \ts .
\]
where $\zeta$ is Riemann's zeta function. \hfill$\Box$
\end{lemma}
For a discussion of the identities from Lemma~\ref{dirichFix} in a
different context, and for interesting recent development on Dirichlet
series as orbit counting generating functions, we refer the reader to
\cite[Secs.~ 2 and 10]{finiteCombRank}.  We now turn our attention to
the structure of $\fix$ and $\orb$ as linear mappings on $\mathcal{A}$.

\section{Spectral properties}

When studying $\fix$ on $\mathcal{A}$,
it is natural to ask for eigenvalues and  eigenvectors.  
Let $x$ be an arithmetic function that satisfies the fixed
point condition $\fix(x)_n = \sum_{d|n} d\, x_d = x_n$ for all
$n\in\NN$.  Solving for $x_n$ yields the recursive relation
\begin{equation}\label{recursion}
    x_n = \frac{1}{1-n}\sum_{n>d | n} d\, x_d
\end{equation}
for $n> 1$, so that all $x_n$ are uniquely determined from $x^{}_{1}$.
The $x_n$ depend linearly on $x_1$, which is the only degree of
freedom in solving \eqref{recursion}, wherefore the eigenspace for the
eigenvalue $1$ is one-dimensional.  Setting $x_1 = a_1 = 1$, we denote
the resulting solution of \eqref{recursion} by $ a =
(a^{}_1,a^{}_2,a^{}_3,\ldots)$.  Since $\fix$ maps multiplicative
sequences to multiplicative sequences, it is an interesting question
to what extent this is reflected by $a$.
\begin{lemma}\label{multiplicativity}
     Let $a$ satisfy $\fix(a) = a$, with $a^{}_1 = 1$.  The eigenvector
    $a$ is multiplicative, so that $a^{}_{nm} = a^{}_n\cdot a^{}_m$ for
    all $n,m \geq 1$ with $(m,n) = 1$.
\end{lemma}
\begin{proof}
  In order to show multiplicativity by induction on the index, we
  first note that the claim is true for trivial decompositions such as
  $n\cdot 1$, because $a^{}_1 = 1$.  In particular, it is then true
  for all natural numbers $n \leq 5 = 2\cdot 3 - 1$.  Let $(m,n) = 1$
  with $mn > 1$ and suppose the $a_k$ satisfy the multiplicativity
  property for all $k \leq mn -1$ (so that $a^{}_k = a^{}_{d}\,
  a^{}_{d\ts '}$ for all $d,d'$ with $(d,d\ts ') = 1$ and $d\, d\ts ' =
  k$).  This implies
\begin{eqnarray*}
     a^{}_{mn} &=& \frac{1}{1-mn}\sum_{\substack{d|mn \\d < mn}} 
       d\, a^{}_d = \frac{1}{1-mn}\sum_{\substack{d=d_md_n\\ d_m|m,\ d_n|n
       \\ d_md_n < mn}}(d_m a^{}_{d_m})\, (d_n a^{}_{d_n})\\
    &=&\frac{1}{1-mn}\Biggl(\Bigl(\underbrace{\sum_{d|m}d\ts  a^{}_d}_{=a^{}_m}
	\underbrace{\sum_{d\ts ' |n}d\ts\ts ' a^{}_{d\ts '}}_{=a^{}_n}\Bigr)- 
       mn\,a^{}_m a^{}_n \Biggr) = a^{}_m\, a^{}_n \ts .
\end{eqnarray*}
This argument shows the induction step from $mn - 1$ to $mn$. In general,
the step from $N$ to $N+1$ is trivial when $N+1$ is prime, and of the above
form otherwise, which completes the induction.
\end{proof}
For the remainder of this section, $a$ always denotes the unique
solution of $\fix(a) = a$ with $a_1 = 1$.  The recursion in
\eqref{recursion} runs over all proper divisors of $n$ and yields an
explicit expression for $a^{}_{p^r}$ with $p$ prime and $r\in\NN$.  By
multiplicativity, this extends to a closed formula for $a_n$.
\begin{prop}\label{productRepres1}
  Let $n > 1$ with prime decomposition $n = p_1^{r_1}\cdot\ldots\cdot
  p_s^{r_s}$. Then, the $n$-th entry of the eigenvector $a^{}$ of
  Lemma~\ref{multiplicativity} reads
\[
     a_n = \prod_{k=1}^{s} \prod_{\ell=1}^{r_k}
     \frac{1}{1-p_k^\ell} \ts .
\]
\end{prop}
\begin{proof}
  Observing that $a^{}_p = \frac{a^{}_1}{1-p}$, induction on the
  exponent $r$ leads to $a^{}_{p^r} = \prod_{\ell = 1}^{r} (1 -
  p^{\ell})^{-1}$, from which the statement follows by
  Lemma~\ref{multiplicativity}.
\end{proof}

The sequence of denominators is now entry \texttt{A153038} of the
Online Encyclopedia of Integer Sequences (OEIS) \cite{OEIS}.  In
particular, $\abs{a_n} \leq 1$ for all $n$, and
Proposition~\ref{productRepres1} suggests that the $a^{}_n$ become
`asymptotically small' in some sense, as $n\to\infty$.  Indeed, with
\[
   \ell^{\ts q} = \bigl\{f\in\CC_{}^{\NN} \, \big|
   \sum_{n=1}^{\infty}|f_n|^q < \infty \bigr\} ,
\]
 one obtains the following result.
\begin{prop}\label{ell1pluseps}
  The eigenvector $a$ is an element of $\ell^{1+\varepsilon}$ for all
  $\varepsilon > 0$.
\end{prop}
\begin{proof}
  Substituting $p^k -1 = (p-1)\sum_{j=0}^{k-1}p^j$ in the denominator
  of $\abs{a_{p^r}}$ leads to
\[
     \prod_{k=1}^{r} (p^k -1 )=
     (p-1)^r\prod_{k=1}^{r}{\sum_{j=0}^{k-1}p^j}\geq (p-1)^r
     p^{\sum_{k=1}^{r}{(k-1)}} = (p-1)^r p^{r(r-1)/2}.
\]
Since $(p-1)^r p^{r(r-1)/2}\geq p^r$ for $r>2$ or $p>2$, the
inequality $\abs{a_{p^r}}\leq \frac{1}{p^r}$ holds in these
cases. Besides, one has
$\abs{a_p}=\frac{1}{p-1}=\frac{p}{p-1}\cdot\frac{1}{p}$ for every
prime $p$.  The remaining term $a_{2^2} = \frac{4}{3}\frac{1}{2^2}$
introduces an additional factor of $c = \frac{4}{3}$ in the estimate
of $\abs{a_n}$, to account for the case when $2^2$ is the highest
power of $2$ that divides $n$. For $n=p_{1}^{r^{}_{1}}\cdot\ldots\cdot
p_{s}^{r^{}_{s}}$, this gives
\[
     \abs{a_n} \le \frac{c}{n} 
     \prod_{\{i\mid r_i = 1 \}}\frac{p_i}{p_i-1}.
\]
This product can be related to Euler's $\phi$-function,
\[
     \phi(n) = n\prod_{p|n}\left(1-\frac{1}{p}\right) =
     n\prod_{p|n}\left(\frac{p-1}{p}\right) \leq n\prod_{\substack{p|n\\p^2
     \nmid\ts n}}^{r}\left(\frac{p-1}{p}\right),
\]
giving $\prod_{\{i\mid r_i = 1 \}}\frac{p_i}{p_i-1} 
\leq \frac{n}{\phi(n)}$. Using
\[
      \limsup_{n\to\infty} \frac{n}{\phi(n)\log(\log n)} = e^{\gamma},
\]
where $\gamma = \lim_{n\to\infty}\left(1+ \frac{1}{2} + \cdots +
  \frac{1}{n} - \log{n}\right)$ denotes Euler's constant, compare
\cite[Thm.~328]{HardyWright}, we obtain
\[
\abs{a_n}\leq \frac{C\cdot \log(\log(n))}{n}\ 
\]
for some constant $C$, which implies
$\sum_{n=1}^{\infty}\abs{a_n}^{1+\varepsilon} <\infty$.
\end{proof}
In particular, $a\in\ell^2$, and (numerically) one finds
$\sum_{n=1}^{\infty} \lvert a_n \rvert^2 \approx 2.99635 < 3$.  
Since $\sum_{p} \frac{1}{p}$ diverges and $\abs{a^{}_p} = \frac{1}{p-1}
> \frac{1}{p}$, it is clear that the statement of
Proposition~\ref{ell1pluseps} is not true for $\varepsilon = 0$.
\begin{lemma}\label{dirichConvergence}
  Let $D_a(s)$ be the Dirichlet series of the eigenvector $a$, with $s
  = \sigma + i t$.
\begin{itemize}
\item[\rm (i)] $\diri{a}{s}$ \text{converges absolutely in} $\,\mathcal{S} =
  \{s\in\CC\mid\sigma > 1\}$.\label{absConvSet} \vspace{1.5mm}
\item[\rm (ii)] $\lim_{\sigma\to\infty}\diri{a}{\sigma + it} = 1$,
  \text{uniformly in} $t\in\RR$.
\end{itemize}
\end{lemma}
\begin{proof} 
  We employ the methods of \cite[Chapter 11]{apostol}.  The first
  claim follows from the absolute convergence of $\zeta(\sigma)$ for
  $\sigma > 1$, because $\sum_{m}\abs{\frac{a_m}{m^s}}\leq
  \sum_{m}\frac{\abs{a_m}}{m^{\sigma}}\leq \zeta(\sigma)$.  

For (ii), observe that $a_1=1$, wherefore one has
\[
   \abs{\diri{a}{s}-1}\leq \sum_{m=2}^{\infty}\frac{1}{m^{\sigma}}
   = \zeta(\sigma)- 1,
\]
which tends to $0$ (independently of $t$) as $\sigma\to\infty$. 
\qedhere
\end{proof} 
Using the fixed point relations $\fix(a) = a$ and $\orb(a) = a$, the
first equation of Lemma~\ref{dirichFix} becomes a `recursion' that can be
solved for an explicit expression in terms of the Riemann zeta
function.
\begin{theorem}\label{dirich_a}
The Dirichlet series of the fixed point $a$ satisfies
\[
    \diri{a}{s} = \prod_{\ell=1}^{\infty} \frac{1}{\zeta(s+\ell)} =
    \prod_{p}\prod_{\ell\geq 1}\left(1-\frac{1}{p^{s+\ell}}\right),
\]
where the infinite product converges on the set\/ $\mathcal{S}$ from
Lemma~\ref{dirichConvergence}.
\end{theorem}
\begin{proof}
  Using Lemma \ref{dirichFix}, one inductively obtains $ \diri{a}{s}
  = \frac{\diri{a}{s+m}}{\prod_{\ell=1}^{m}\zeta(s+\ell)} $ for
  arbitrary $m\in\NN$.  For $m\to\infty$, the numerator converges to
  $1$, since
\[
\diri{a}{s+m} = a_1 + \sum_{n\geq 2}\frac{a_n}{n^s}\frac{1}{n^m}\leq
a_1 + \frac{1}{2^m}\diri{a}{s}.
\]
Due to the asymptotic behaviour $\zeta(\sigma + i t) = 1+
\mathcal{O}(1/2^{\sigma})$ for $\sigma\to\infty$, the product in the
denominator of the resulting expression $D_a(s) =
\prod_{\ell=1}^{\infty} \frac{1}{\zeta(s+\ell)}$ converges as well,
giving the first part of the claim.  The second equality follows from
the Euler product representation $ \zeta(s+\ell)^{-1} = \prod_{p} (1 -
p^{-(s+\ell)}) $ for $\sigma > 1$, in which the order of taking
products over $p$ and $\ell$ may be exchanged due to absolute
convergence.
\end{proof}
\begin{rem}
  Since $a^{}_1\not= 0$, $a$ admits a Dirichlet inverse. Its terms can
  be calculated via to the general recursion formula, compare
  \cite[Thm.~2.8]{apostol}.  The Dirichlet inverse satisfies
  $D^{}_{f^{-1}} = 1/D^{}_{f}$, wherefore $b := a^{-1}$ has the
  Dirichlet series $D^{}_{b} = \prod_{\ell \geq 1} \zeta(s + \ell).$
  Determining the coefficients of this series leads to an analogue of
  Proposition~\ref{productRepres1}; the $n$-th entry of $b$ is
\[
       b^{}_n = \prod_{j=1}^{s} \prod_{k=1}^{r_j} \frac{p_j^{k-1}}{p_j^{k}-1}
       =\prod_{j=1}^{s} \frac{p_j^{\frac{1}{2} r_j(r_j-1)}}{\prod_{\ell =
        1}^{r_j} (p_j^{\ell}-1)},
\]
with $n = p_1^{r^{}_1}\cdot\ldots\cdot p_s^{r^{}_s}$ as before; see
\cite{NN} for details. The first few terms of $a$ and $b$ read
\[
  a = \Bigl( 1,-1,-\frac{1}{2},\frac{1}{3},-\frac{1}{4},
     \frac{1}{2},-\frac{1}{6},-\frac{1}{21}, \ldots \Bigr)
  \quad \text{and} \quad  b = \Bigl(
  1,1,\frac{1}{2},\frac{2}{3},\frac{1}{4},\frac{1}{2},
  \frac{1}{6},\frac{8}{21}, \ldots \Bigr).
\]
Although it appears that $a$ and $b$ share the same denominator
sequence, this is not true when the terms are represented as reduced
fractions, because cancellation can occur for the $b_n$.  The first
instance of this phenomenon is $b_{12} = \frac{1}{3}$, whereas
$a_{12}=-\frac{1}{6}$.  
\end{rem}

Returning to the task of identifying all eigenvalues and their
corresponding eigenspaces, we consider the equation $x = \lambda
\cdot\fix(x)$ for $\lambda \not= 1$ or, equivalently,
\[
   (\lambda - n) \, x^{}_n = 
   \lambda \sum_{n > d|n}d\, x^{}_d  
\]
for all $n\in\NN$.  Since $\fix(x)^{}_1 = x^{}_1$ by definition, the
first term of such an eigenvector has to be $0$.  In general, if
$x^{}_d = 0$ for all $d|n$ with $d<n$, the $n$-th term $x^{}_n$ will
vanish as well, unless $\lambda = n$.  Thus, the first non-zero term
is $x^{}_m$ where $m = \lambda$, and all eigenvalues are thus natural
numbers.  For fixed $1 < m\in\NN$, the recursion analogous to the case
$m=1$ given in \eqref{recursion} is
\begin{equation}
       x^{}_n = \frac{1}{m  - n}\sum_{\substack{n>d|n}}d\, x^{}_d\ 
       \qquad \text{(for}\quad n\not= m \text{)},
\end{equation}
which shows that $x^{}_n \not= 0$ if and only if $n = km$ for some
$k\geq 1$.  The recursion for the non-vanishing terms
simplifies to
\[
     x_{km}^{} = \frac{1}{m(1-k)}\sum_{k>d|k}^{} d\ts m\, x^{}_{dm}
           = \frac{1}{1-k}\sum_{k>d|k}^{} d\, x^{}_{dm}
\]
for $k>1$,where the only free parameter is now the first non-zero term
$x_m$, indicating a one-dimensional eigenspace.  Choosing $x_m =
a^{}_1 = 1$, the resulting eigenvector $a^{(m)}$ with eigenvalue $m$
can be expressed in terms of $a = a^{(1)}$:
\begin{equation}
    a^{(m)} (n) = \begin{cases}
    a\left(\frac{n}{m}\right), & \text{if } m|n ; \\
           0,      & \text{otherwise}.
\end{cases}
\end{equation}
In particular, $a^{(m)}\in\ell^{1+\varepsilon}$ for all $\varepsilon >
0$ and all $m\geq 1$, where $a^{(m)}$ and $a$ clearly have the same
norm.  Being eigenvectors for different eigenvalues, the $a^{(m)}$ are
linearly independent. The next proposition states that they even
generate the space of arithmetic functions in the sense of formal
linear combinations.
\begin{prop}
  The eigenvectors of\/ $\fix$ form a basis of $\mathcal{A}$, so
  that $\left<a^{(m)}\mid m\in\NN\right>_{\CC} = \mathcal{A}$,
  and the representation of $f\in\mathcal{A}$ as a formal linear
  combination of the $a^{(m)}$ is unique.
\end{prop}
\begin{proof}
  Let $f\in\mathcal{A}$, $a = a^{(1)}$ and $b := (a^{})^{-1}$, and
  set $\alpha = f \ast b$.  Then, $\alpha * a^{} = (f*b)*a = f$.  On
  the other hand, one has
\[
        \bigl(\alpha* a \bigr) (k) =  \sum_{d|k} \alpha^{}_{d} \, a^{(d)} (k)
        =  \sum_{m\geq 1}  \alpha^{}_{m}\, a^{(m)}(k)
\]        
  for all  $k\in\NN$, so $f\in \left<a^{(m)}\mid m\in\NN\right>_{\CC} $.
 Uniqueness of the linear combination follows  from the uniqueness 
 of the Dirichlet inverse.
\end{proof}
The spectrum of an operator $T\!: \, \mathcal{A}\to\mathcal{A}$ is 
defined by
\[
     \mathrm{spec}(T) := \{\lambda\in\CC\mid T-\lambda\,\id 
     \text{ is not invertible}\, \} \ts .
\]
Note that $\fix$ and $\orb$ could be investigated as operators on
$\ell^2$, restricting their domains to appropriate invariant
subspaces.  But since we need a more general setting later when
considering their behaviour under iteration, we work with the spectrum
as a purely algebraic notion.
\begin{theorem}
The spectra of\/ $\fix$ and $\orb$ are
\[
   \mathrm{spec}(\fix) = \NN  \quad and \quad 
    \mathrm{spec}(\orb) =
          \left\{\frac{1}{N} \mid N\in\NN \right\}\ts .
\]
\end{theorem}
\begin{proof}
  Obviously, $\fix(f) = \lambda f$ implies $f=0$ if and only if
  $\lambda\in\CC\setminus\NN$, so $\fix - \lambda\,\id$ is
  injective for these $\lambda$.  Surjectivity can be seen by
  constructing a preimage $f\in\mathcal{A}$ of an arbitrary
  $g\in\mathcal{A}$ under $\fix-\lambda\, \id$.  To this end, let
  $f_1 = \frac{g^{}_1}{1-\lambda}$ and, having already defined
  $f_1,\ldots,f_{n-1}$ for some $n>1$, set
\[
          f^{}_n = \frac{1}{n-\lambda} \Bigl(g^{}_n - \sum_{n>d|n}
          d\,  f^{}_d\Bigr).
\]
When $\lambda\notin\NN$, the resulting arithmetic function $f$ is a
term-wise well-defined element of $\mathcal{A}$ with $\fix(f) -
\lambda f = g$.  In summary, $\fix-\lambda\,\id$ is invertible if and
only if $\lambda\in\CC\setminus\NN$.  The spectrum of $\orb$ can be
derived from that of $\fix$ by the equivalence
\[
      \fix(a^{(N)}) = N a^{(N)} \, \Longleftrightarrow \;
      \orb(a^{(N)}) = \frac{1}{N} \, a^{(N)},
\]
which follows from  $\fix\circ\orb = \orb\circ\fix = \id$.
\end{proof}

\section{Iterations of the operator $\fix$}

Puri and Ward \cite{PuriWard} raised the question of the nature of the
orbits under $\fix$. In particular, the asymptotic properties of such
orbits are sought.  In our terminology, this is to ask how
elements of $\mathcal{A}$ behave under iteration of the operator
$\fix$.  Starting from $\delta_{1} = (1,0,0,\ldots)$, the first
iterates are sequences \texttt{A000007}, \texttt{A000012},
\texttt{A000203}, \texttt{A001001} and \texttt{A038991} --
\texttt{A038999} of the OEIS, see \cite{PuriWard, OEIS}.  Here, we
start from the sequences $\delta_m$, $m\geq 1$, defined via $\delta_m
(n) = \delta_{m,n}$ for all $n$, and give term-wise exact
representations of $\fix^n(\delta_m)$, from which the corresponding
quantities of general starting sequences can later be constructed.

Let $y^{(0)} = \delta_1$ and $y^{(n)} := \fix(y^{(n-1)})$ for $n\geq
1$, so that
\[
     y^{(n)}(m) = \bigl(\fix(y^{(n-1)})\bigr)_m =
     \bigl(\fix^n(y^{(0)})\bigr)_m
\]
holds for all $m\ge 1$.
\begin{prop}\label{y(primePower)}
      Let $p$ be a prime and $r\in\NN$. Then
\[
     y^{(n)}(p^r) = \gabi{n+r-1}{r}{p} = 
     \prod_{i=0}^{r-1}\frac{1-p^{n+r-1-i}}{1-p^{i+1}},
\]
      where $\sga{n}{r}{q}$ denotes the Gaussian or $q$-binomial 
      coefficient.
\end{prop}
\begin{proof}
  The claim can be verified by induction in $n$. By definition of
  $\fix$, we have
\begin{align}\label{E:sumBin}
  y^{(n+1)}(p^r) &= \fix(y^{(n)})(p^r) =
  \sum_{d|p^r}d\cdot y^{(n)}(d)= \sum_{k=0}^{r}p^k \cdot y^{(n)}(p^k) 
   \notag \\
  &= 1+ p\gabi{n}{1}{p}+\sum_{k=2}^{r}p^k \cdot y^{(n)}(p^k)  \ts .
\end{align}
Applying \textit{Pascal's identity} for Gaussian binomials,
\[
   \gabi{m}{\ell}{q} = q^r \gabi{m-1}{\ell}{q} + 
   \gabi{m-1}{\ell-1}{q},
\]
to the $k$-th summand in \eqref{E:sumBin} yields
\[
    p^k\gabi{n+k-1}{k}{p}=\gabi{n+k}{k}{p}-
    \gabi{n+k-1}{k-1}{p}.
\]
Plugging this and the relation
$\sga{n}{1}{p}=\sum_{k=0}^{n-1} p^k$ into \eqref{E:sumBin},
we obtain
\[
    y^{(n+1)}(p^r) = \gabi{n+1}{1}{p} + 
    \sum_{k=2}^{r}\gabi{n+k}{k}{p}-\gabi{n+k-1}{k-1}{p}  
     =\gabi{n+r}{r}{p}
\]
which completes the proof.
\end{proof}
\begin{corollary}\label{y(m)}
       For $1\not= m\in\NN$ with prime decomposition $m =
       p_1^{r^{}_1}p_2^{r^{}_2}\cdot\ldots\cdot p_k^{r^{}_k}$, one has
\[
     y^{(n)}(m)=y^{(n)}(p_1^{r^{}_1})\, 
     y^{(n)}{(p_2^{r^{}_2})}\cdot\ldots\cdot y^{(n)}(p_k^{r^{}_k}) 
     =\prod_{j=1}^{k} \prod_{i=0}^{r_j -1}
       \frac{1-p_j^{n+r_j-1-i}}{1-p_j^{i+1}}.
\]
\end{corollary}
\begin{proof}
  It is easy to check that $\fix$, and hence all iterates $\fix^n$
  with $n\geq 1,$ preserve multiplicativity.  Thus, in view of
  Proposition~\ref{y(primePower)}, the claim follows from $\delta_1$
  being multiplicative.
\end{proof}
\begin{rem}
  A different approach to calculating $y^{(n)}$ exploits the fact that
  its Dirichlet series $D^{}_{y^{(n)}}(s)$ is known from the context
  of counting sublattices of $\ZZ^n$.  Let
\[
    h^{}_n(m) = \big\lvert\{\varLambda \mid \varLambda
    \text{ is a sublattice of $\ZZ^n$ of index $m$}\}\big\rvert
\]
and $D_{h_n}(s) = \sum_{m\geq 1} \frac{h_n(m)}{m^s}$.
Proposition A.1 of \cite{coincApp} states that
\[
   h_{n}(m) = \sum_{d_1\cdot\ldots\cdot d_n = m} 
   d_1^{\ts 0}\cdot d_2^{\ts 1}\cdot\ldots\cdot 
   d_n^{\ts n-1}    \quad \text{and}\quad 
   D^{}_{h_n}(s) = \zeta(s)\cdot D^{}_{h_{n-1}}(s-1),
\]
where the first sum runs over all $n$-tuples $(d_1,\ldots, d_n)$ of
positive integers with $d_1\cdot\ldots\cdot d_n = m$.  In particular,
$D^{}_{h_1}(s) = \zeta(s) = D^{}_{y^{(1)}}(s)$.  On the other hand,
Lemma~\ref{dirichFix} gives $D^{}_{y^{(n+1)}}(s) = \zeta(s)
D^{}_{y^{(n)}}(s-1)$, so $D^{}_{y^{(n)}}(s) = D^{}_{h_n}(s)$ for all
$n\geq 1$ and thus $h_n(m) = y^{(n)} (m)$ for all $n,m\geq 1$.
For $m = p_1^{r^{}_1}\cdot\ldots\cdot
p_{k}^{r^{}_k}$, the coefficient $h_n(m)$ can be written as
\[
       h_n(m) = \prod_{j=1}^{k}\prod_{i=1}^{r_j}
         \frac{p_j^{n+i-1}-1}{p_j^i-1}
      = \prod_{j=1}^{k}\gabi{n+r_j-1}{r_j}{p_j} \ts ,
\]
as shown in \cite{Gruber} and \cite{Zou}, leading to the result of
Corollary~\ref{y(m)}. This also provides a concrete interpretation
of the iteration sequences mentioned earlier.
\end{rem}

Similarly to calculating $\fix$-eigenvectors $a^{(m)}$ for $m > 1$,
the sequences arising from $\fix$-iterations on starting sequences
$\delta_m$ with $m > 1$ can be related to the case $m=1$. From
$\bigl(\fix(\delta_m)\bigr)_k = m\cdot \delta_{k,jm}$ for some
$j\in\NN$, one can conclude inductively
\[
     \bigl(\fix^{n}(\delta_m)\bigr)_k = 
      \begin{cases}  m^{n}\! \cdot y^{(n)}\bigl(\tfrac{k}{m} \bigr),       
           &\text{if $m | k$}; \\
        0, &\text{otherwise.}  \end{cases}
\]
Since arbitrary elements from $\mathcal{A}$ can be written as
complex linear combinations of the $\delta_m$, we find the following
behaviour of general arithmetic sequences under $\fix$-iterations.
\begin{lemma}\label{nthIterationGeneral}
       Let $f = (f_1,f_2,\ldots)\in\mathcal{A}$. Then, for all $M\geq 1$,
\[
     \bigl(\fix^n(f)\bigr)_M = \sum_{k=1}^{M}
     f_k\, \bigl(\fix^n(\delta_k)\bigr)_M =
     \sum_{k|M} f_k\, k^{n}  y^{(n)}\bigl(\tfrac{M}{k}\bigr).
\]
As a convolution product, this is
\[
     \fix^n (f) = \bigl(f\cdot N^n \bigr) * y^{(n)},
     \quad \text{for all $n\in\NN$},
\]
where $N^n$ denotes  the arithmetic function defined by $m\mapsto m^n$.
\end{lemma}
\begin{proof}
  Note first that, for all $M\geq 1$, the value $\fix(f)_M$
  depends on the $f_n$ with $n|M$ only.  Each single term of
  $\fix(f)_M$ can thus be obtained by applying $\fix$ to the
  finite linear combination $\sum_{k=1}^M f_k \delta_k$, resulting in
  a component-wise well-defined arithmetic function.  The last
  identity follows from the definition of the convolution product.
\end{proof}
Given the unbounded spectrum of $\fix$, it is not surprising that the
image sequences $\fix^{n}(f)$ of a non-negative $f$ component-wise
tend to infinity when $n\to\infty$.  Nevertheless, the `quotient
sequences' $\left(\frac{\fix^{n+1}(f)_M}{\fix^{n}(f)_M}\right)_n$ have
a well-defined convergence behaviour for each $M$ as follows.

\begin{theorem}
  Let $f$ be a non-negative arithmetic function and let $M\in\NN$ be
  such that $\bigl(\fix^r(f)\bigr)_M \not= 0$ for some $r\in\NN$.
  Then,
\[
      \lim_{n\to\infty} \frac{\bigl(\fix^{n+1}(f)\bigr)_{M}}
        {\bigl(\fix^{n}(f)\bigr)_{M}} = M \ts .
\]
\end{theorem}
\begin{proof}
  Expanding the quotient in question yields
\[
   \frac{\bigl(\fix^{n+1}(f)\bigr)_{M}}{\bigl(\fix^{n}(f)\bigr)_{M}} 
   \, = \, \sum_{d|M} d\
   \frac{\bigl(\fix^{n}(f)\bigr)_{d}}{\bigl(\fix^{n}(f)\bigr)_{M}} 
   \, = \, M + \!\! \sum_{M> d|M} d\
   \frac{\bigl(\fix^{n}(f)\bigr)_{d}}{\bigl(\fix^{n}(f)\bigr)_{M}} .
\]
The number of terms in the last sum being independent of $n$, it
suffices to show that, for each divisor $d$ of $M$, the corresponding
summand approaches $0$ as \text{$n\to\infty$.}  Let $d^{}_1=1, \ldots,
d^{}_R = d$ denote the divisors of $d$. Then, using
Lemma~\ref{nthIterationGeneral}, one gets
\[
  \frac{\bigl(\fix^{n}(f)\bigr)_{d}}{\bigl(\fix^{n}(f)\bigr)_{M}} 
  \, = \, \sum_{i=1}^R \frac{f_{d_i}\
  d_i^n\ y^{(n)}\bigl(\frac{d}{d_i}\bigr)}{\sum_{k|M}f_k\, {k}_{}^{n}
  y^{(n)}\bigl(\frac{M}{k}\bigr)} \,\leq\, \sum_{i=1}^R \frac{d_i^n\
  f_{d_i} y^{(n)}(d/d_i)}{d_i^n\ f_{d_i} y^{(n)}(M/d_i)}  \, = \,
  \sum_{i=1}^{R}\frac{y^{(n)}(d/d_i)}{y^{(n)}(M/d_i)} \ts .
\]
This upper bound can be seen to converge to $0$ by splitting numerator
and denominator of each summand in the last sum into the factors that
arise from the prime decomposition of $d/d_i$ and $M/d_i$.  For $r >
s$, the quotient
\begin{equation}\label{singleFactor}
  \frac{y^{(n)}(p^s)}{y^{(n)}(p^r)}=
   \prod_{i=0}^{s -1}\frac{1-p^{n+s-1-i}}{1-p^{n+r-1-i}}
   \prod_{i=s}^{r-1}\frac{1-p^{i+1}}{1-p^{n+r-1-i}}
\end{equation}
converges to $0$ as $n\to\infty$, since the first product tends to
$p^{s-r} > 0$, while the second tends to $0$.  Being composed 
of factors such as \eqref{singleFactor}, the quotient 
$\bigl(\fix^n(f)\bigr)_{d} \, /\bigl(\fix^n(f)\bigr)_M $ tends to
$0$ for all $d\ts |M$ as ${n\to\infty}$, which proves the claim.
\end{proof}
\begin{rem}
  Unlike $\fix$, the operator $\orb$'s spectrum is bounded by $0$ from
  below and by $1$ from above.  One would expect $\orb^n(f)$ to
  converge to an element of the eigenspace for the largest eigenvalue
  ($\lambda = 1$), hence to a scalar multiple of $a^{(1)}=a$. Indeed,
  employing the eigenvector basis representation $f = \sum_{m\geq 1}
  \alpha^{}_m a^{(m)}$, one obtains
\[
  \lim_{n\to\infty} (\orb^n f)^{}_M = \lim_{n\to\infty}
  \sum_{m=1}^{M} \alpha^{}_m \orb^n(a^{(m)})_M = 
  \lim_{n\to\infty}\sum_{m=1}^{M}\alpha^{}_m \frac{1}{m^n} a^{(m)}_{M}
  = \alpha^{}_1 a^{(1)}_M
\]
for all $M\in\NN$. In other words, $\orb^n(f) 
\xrightarrow{\; n\to\infty \; } f^{}_1\, a$, 
for all $f\in\mathcal{A}$.
\end{rem}

\section{Concluding remarks}
The eigenvector $a$ has an interesting consequence in the setting of
dynamical (or Artin-Mazur) zeta functions, where one considers the
generating function
\[
\exp\left(\sum_{m=1}^{\infty} \frac{f_m}{m}\ z^m\right) =
\prod_{m=1}^{\infty} (1-z^m)^{-\orb(f)^{}_m},
\]
usually for arithmetic functions $f$ that count (isolated) fixed
points of a discrete dynamical system, see \cite{Fel} for background
material.  Using $a$ together with $a=\orb(a)$, one finds the product
identity
\[
    \exp\left(\sum_{m=1}^{\infty} \frac{a_m}{m}\ z^m\right) =
    \prod_{m=1}^{\infty} (1-z^m)^{-a_m}
\]
or, equivalently, the series identity
\[
    \sum_{m=1}^{\infty}\frac{a_m}{m}\, z^m = 
    -\sum_{m=1}^{\infty}a_m\,\log(1-z^m),
\]
with absolute convergence for $\abs{z} < 1$.  It would be interesting
to know whether this gives rise to another interpretation of the
meaning of the eigenvector $a$\ in this setting.

\subsection*{Acknowledgements}
It is our pleasure to thank Tom Ward for his interest in this work and
to John A.\ G.\ Roberts and Johan Nilsson for useful hints on the
manuscript.  This work was supported by the German Research Council
(DFG), within the CRC 701.



\bigskip

\begin{thebibliography}{xx}
\bibitem{apostol} 
T.~M.~Apostol,
\textit{Introduction to Analytic Number Theory}, 
Springer, New York (1976).

\bibitem{coincApp} 
M.~Baake, 
\textit{Solution of the coincidence problem in dimensions $d\le 4$}, 
in:\ \textit{The Mathematics of Long-Range Aperiodic Order}, 
ed.\ R.\ts V.\ Moody, NATO-ASI Series C\ts 489, 
Kluwer, Dordrecht (1997), pp.\ 9--44; rev.\ version, 
\texttt{arXiv:math.MG/0605222}.

\bibitem{finiteCombRank} 
G.~Everest, R.~Miles, S.~Stevens and T.~Ward, 
\textit{Dirichlet series for finite combinatorial rank dynamics}, 
preprint \texttt{arXiv:0705.1067}.

\bibitem{Fel}
A.~Fel'shtyn,
\textit{Dynamical Zeta Functions, Nielsen Theory and
Reidemeister Torsion},
Memoirs AMS, vol.\ 147, no.\ 699, AMS, Providence, RI (2000).

\bibitem{Gruber} 
B.~Gruber, 
\textit{Alternative formulae for the number of sublattices}, 
Acta Cryst. \textbf{A53} (1997), 807--808.

\bibitem{HardyWright} 
G.~H.~Hardy and E.~M.~Wright, 
\textit{An Introduction to the Theory of Numbers}, 
5th ed., Clarendon Press, Oxford (1979).

\bibitem{mossDiss} 
P.~Moss, 
\textit{Algebraic realizability problems}, 
PhD thesis, University of East Anglia, Norwich (2003).

\bibitem{NN} 
N.~Neum\"arker, 
\textit{Orbitstatistik und relative Realisierbarkeit}, 
Diplomarbeit, Universit\"at Bielefeld (2007).

\bibitem{purisDiss} 
Y.~Puri, 
\textit{Arithmetic and growth of periodic orbits}, 
PhD thesis, University of East Anglia, Norwich (2000).

\bibitem{PuriWard} 
Y.~Puri and T.~Ward,  
\textit{Arithmetic and growth of periodic orbits}, 
J.\ Integer Seq.\ \textbf{4}, article 01.2.1 (2001).

\bibitem{OEIS} 
N.~J.~A.~Sloane,
\textit{The Online Encyclopedia of Integer Sequences}, \newline
\texttt{www.research.att.com/$\sim$njas/sequences/}

\bibitem{Zou} 
Y.~M.~Zou, 
\textit{Gaussian binomials and the number of sublattices}, 
Acta Cryst. \textbf{A62} (2006), 409--410;
\texttt{arXiv:math/0610684}.

\end{thebibliography}
\end{document}